\documentclass[11pt,a4paper]{article}

\usepackage{newpxtext} 
\usepackage{newpxmath} 

\usepackage{geometry}
\usepackage[utf8]{inputenc}
\usepackage[T1]{fontenc}
\usepackage{amsmath}

\usepackage{amsthm}
\usepackage{graphicx}
\usepackage{float}
\usepackage{subcaption}
\usepackage{booktabs}
\usepackage{url}
\usepackage{algorithm,algorithmic}
\usepackage[legacycolonsymbols]{mathtools}

\usepackage[most]{tcolorbox}
\newtcolorbox{mybox}[1]{
  colback=white, colframe=black, coltitle=black,
  fonttitle=\bfseries, attach title to upper, after title={:\ },
  sharp corners, boxrule=0.5pt, left=10pt, right=10pt, top=10pt, bottom=10pt,
  title=#1
}

\usepackage{xcolor}
\definecolor{arxivblue}{rgb}{0.0, 0.0, 0.5} 
\usepackage[colorlinks=true, linkcolor=arxivblue, citecolor=arxivblue, urlcolor=arxivblue]{hyperref}

\newtheorem{theorem}{Theorem}
\newtheorem{lemma}{Lemma}

\title{\textbf{Extension of the safeguarding stepsize interval in Adaptive Gradient Descent}}

\author{
  Saneatsu Kagawa\thanks{Graduate School of Informatics, Kyoto University, Japan, \texttt{kagawa.saneatsu.67k@st.kyoto-u.ac.jp}}
  \and
  Nobuo Yamashita\thanks{Graduate School of Informatics, Kyoto University, Japan, \texttt{nobuo@i.kyoto-u.ac.jp}}
}
\date{June 3, 2026}

\begin{document}
\maketitle

\begin{abstract}
In this paper, inspired by the idea of Adaptive Gradient Descent (AdGD) by Malitsky and Mishchenko, we propose a method that adaptively provides a stepsize interval, which represents the condition for stepsizes ensuring global convergence. 
By projecting the Barzilai Borwein stepsize onto this interval, we can ensure global convergence and expect further acceleration of convergence. 
Furthermore, we propose enlarging the stepsize interval obtained by AdGD and show that global convergence can be guaranteed within this enlarged stepsize interval.
This extension of the interval enables us to adopt the full Barzilai-Borwein (BB) stepsize more frequently.
We report the results of numerical experiments on the stepsize that combines the proposed interval with the BB stepsize.
\end{abstract}

\section{Introduction}
An unconstrained optimization problem is the problem of minimizing an objective function without any constraints and is widely used in various fields such as machine learning and data analysis.
Among them, a problem whose objective function is convex is called an unconstrained convex optimization problem.

When considering solving an unconstrained convex optimization problem, depending on the objective function, the optimal solution cannot always be obtained analytically.
Therefore, in practice, numerical solution methods, called iterative methods, are used.
In these methods, a sequence of points $\{x^k\}$ is generated using a search direction $d_k$ and a stepsize $t_k$ using the following update formula, and its convergence point is taken as the optimal solution:
\[
x^{k+1}=x^k+t_kd_k.
\]
Various iterative methods exist depending on how the search direction is determined.
Among them, a representative method is gradient descent \cite{cauthy}, where the search direction is set in the opposite direction of the gradient vector of the objective function.

In gradient descent, the speed of convergence to the optimal solution depends heavily on how the stepsize is determined.
Therefore, various stepsize calculation methods have been devised with the aim of improving the convergence rate.

First, the simplest one is the fixed stepsize.
This method fixes the stepsize, and since there is no need to compute the stepsize at each step, the computational cost per step can be reduced.
However, when the stepsize is fixed, global convergence to the optimal solution cannot be guaranteed unless the objective function is globally smooth \cite{fixedstepsize1}, so its usefulness is limited to specific functions.
Furthermore, even if it is globally L-smooth, several practical limitations remain.
Specifically, computing the value of stepsize requires explicit knowledge of the Lipschitz constant L, which is often unavailable in practice.
Moreover, if L is too large, the stepsize required to ensure convergence becomes extremely small, which significantly degrades the convergence rate.

As another stepsize determination method is line search.
In contrast to fixed stepsize, line search computes the stepsize at each step.
At this time, it is necessary to compute the stepsize to satisfy specific conditions such as the Armijo condition \cite{Armijo} and the Wolfe condition \cite{wolf}.
This method offers significant advantages; it eliminates the need for the Lipschitz constant L to compute the stepsize, while preventing the stepsize from becoming excessively large or small.
However, on the other hand, it is necessary to evaluate the function value multiple times at each step, so it suffers from the problem that the computational cost becomes very large.

Furthermore, there are other methods such as the Barzilai-Borwein method (BB method) \cite{BB}.
The BB method has a relatively fast convergence rate.
However, its convergence cannot be guaranteed unless the objective function is a strictly convex quadratic function.
Thus, in any case, there are some problems, such as the strictness of the assumptions on the objective function or the largeness of computational cost.

Adaptive Gradient Descent \cite{malitsky2019,malitsky2024} overcomes these drawbacks.
Unlike the fixed stepsize or the Barzilai-Borwein method, Adaptive Gradient Descent has been shown to converge to the optimal solution assuming only that the objective function is locally L-smooth.
Moreover, unlike line search, it does not require the function value to compute the stepsize, and since it can be obtained in a single calculation, the computational cost is small.
However, as a drawback, it cannot be combined with heuristic methods that have fast convergence rates in numerical experiments, such as the Barzilai-Borwein method.

Therefore, the purpose of this report is to overcome the drawbacks of Adaptive Gradient Descent and further accelerate the convergence rate.
Specifically, we find a stepsize interval in which convergence to the optimal solution is guaranteed in Adaptive Gradient Descent, and consider the median of three values: its upper bound, its lower bound, and a stepsize obtained by a method with fast convergence such as the Barzilai-Borwein method, as a new stepsize (a similar method is proposed in \cite{Andreas}).
This ensures convergence to the optimal solution while accelerating the convergence rate.
Furthermore, in this report, we extend the upper bound of such a stepsize interval.
This makes it easier to adopt the stepsize determined by the Barzilai-Borwein method or others, and faster convergence is expected.

The structure of this report is as follows.
In section 2, we present the details of Adaptive Gradient Descent.
In section 3, we present the stepsize interval that ensures global convergence to the optimal solution by using Adaptive Gradient Descent, and further extend its upper bound.
In section 4, we propose a method that combines Adaptive Gradient Descent with other stepsize calculation methods such as the Barzilai-Borwein method based on the extended stepsize interval.
In section 5, we compare the performance of the proposed method and the original Adaptive Gradient Descent through numerical experiments.
Finally, in section 6, we provide a summary of this study and future work.

\section{Adaptive Gradient Descent}
In this section, we present the details of Adaptive Gradient Descent (AdGD) \cite{malitsky2019,malitsky2024}.

Consider the following unconstrained convex optimization problem:
\begin{equation}\label{opt1}
  \min_{x\in\mathbb{R}^d}{f(x)}
\end{equation}
where $f\colon\mathbb{R}^d\rightarrow\mathbb{R}$ is a differentiable convex function.

We consider solving this problem \eqref{opt1} using gradient descent with the stepsize $t_k$.
At this time, the sequence of iterations $\{x^k\}$ at each step of gradient descent is updated by the following formula:
\begin{equation}\label{gd}
  x^{k+1}=x^k-t_k\nabla f(x^k).
\end{equation}

Here, let $x^*$ be a solution of \eqref{opt1}.
Using equation \eqref{gd}, we can see that the following inequality holds:
\begin{align}\label{standard}
  \|x^{k+1}-x^*\|^2&=\|x^k-t_k\nabla f(x^k)-x^*\|^2\nonumber\\
  &=\|x^k-x^*\|^2-2t_k\left\langle \nabla f(x^k),x^k-x^*\right\rangle+t_k^2\|\nabla f(x^k)\|^2\nonumber\\
  &\leq\|x^k-x^*\|^2-2t_k\left(f(x^k)-f(x^*)\right)+t_k^2\|\nabla f(x^k)\|^2
\end{align}
where the last inequality uses Lemma \ref{convex1} in the appendix since $f$ is a convex function.

This inequality is a crucial equation for showing the convergence to a solution.

First, we explain the Adaptive Gradient Descent method proposed in \cite{malitsky2019}.

Let the stepsize $t_k$ be given as follows.
\begin{mybox}{Adaptive Gradient Descent}
\begin{align}
  &L_k=\frac{\|\nabla f(x^k)-\nabla f(x^{k-1})\|}{\|x^k-x^{k-1}\|}\label{locallipschiz}\\
  &\theta_k=\frac{t_k}{t_{k-1}}\label{theta}\\
  &t_k=\min\left\{\sqrt{1+\theta_{k-1}}t_{k-1},\quad\frac{1}{2L_k}\right\}\label{adgdstep}
\end{align}
where $\theta_0$ is an appropriate positive real number.
\end{mybox}

This method is called Adaptive Gradient Descent (AdGD). In this paper, we refer to this method as AdGD-0.

It is known that a solution of \eqref{opt1} can be obtained using AdGD-0.
This is shown in the following Theorem \ref{adgdconv}.
\begin{theorem}(\cite{malitsky2019} Theorem 1.)\label{adgdconv}
  Suppose that $f\colon\mathbb{R}^d\to\mathbb{R}$ is convex with locally Lipschitz gradient $\nabla f$.
  Then $\{x^k\}$ generated by AdGD-0 converges to a solution $x^*$.\qed
\end{theorem}

As can be seen from this theorem, AdGD-0 can ensure convergence to a solution if the objective function is locally smooth, even if it is not globally smooth in general.

A method that analyzes this AdGD-0 in more detail is proposed in \cite{malitsky2024}.

Let the stepsize $t_k$ be given as follows.
\begin{mybox}{Adaptive Gradient Descent-1 (AdGD-1)}
\begin{equation}\label{adgdstep3}
    t_k=\min\left\{\sqrt{1+\theta_{k-1}}t_{k-1},\quad\frac{1}{\sqrt{2}L_k}\right\}
\end{equation}
where $L_k$ is the same as \eqref{locallipschiz} and $\theta_k$ is the same as \eqref{theta}.
\end{mybox}
In this paper, we call this method Adaptive Gradient Descent-1 (AdGD-1).

Similarly to Theorem \ref{adgdconv}, it can be shown that the sequence of iterations $\{x^k\}$ generated by AdGD-1 converges to a solution $x^*$ of \eqref{opt1} \cite{malitsky2024}.
Comparing the stepsize \eqref{adgdstep3} of AdGD-1 with the stepsize \eqref{adgdstep} of AdGD-0, it is clear that the former is larger than the latter.

Furthermore, \cite{malitsky2024} also proposes a method that gives the stepsize $t_k$ as follows.
\begin{mybox}{Adaptive Gradient Descent-2 (AdGD-2)}
\begin{equation}
    t_k=\min\left\{\sqrt{\frac{2}{3}+\theta_{k-1}}t_{k-1},\quad\frac{t_{k-1}}{\sqrt{[2t_{k-1}^2L_k^{2}-1]_+}}\right\}\label{adgdstep2}
\end{equation}
where $L_k$ is the same as \eqref{locallipschiz}, $\theta_k$ is the same as \eqref{theta} and
$[t]_+=\max\{t,0\}$.
\end{mybox}
In this paper, we call this method Adaptive Gradient Descent-2 (AdGD-2).

For this method as well, the sequence $\{x^k\}$ generated by AdGD-2 converges to a solution $x^*$. The following Theorem \ref{adgd2} shows this fact.
\begin{theorem}(\cite{malitsky2024} Theorem 2.)\label{adgd2}
  Let $f\colon\mathbb{R}^d\rightarrow\mathbb{R}$ be convex with locally Lipschitz gradient $\nabla f$.
  Then, the sequence $\{x^k\}$ generated by AdGD-2 converges to a solution $x^*$.\qed
\end{theorem}

Here, we compare the stepsize of AdGD-2 with the stepsize of AdGD-1.
First, comparing the second terms, since
\begin{equation*}
  \frac{t_{k-1}}{\sqrt{[2t_{k-1}^2L_k^{2}-1]_+}}>\frac{t_{k-1}}{\sqrt{[2t_{k-1}^2L_k^{2}]_+}}=\frac{1}{\sqrt{2}L_k}
\end{equation*}
holds, the stepsize of AdGD-2 is larger than that of AdGD-1. On the other hand, comparing the first terms, it is clear that AdGD-1 is larger than AdGD-2.
Therefore, there is no definitive relationship between the magnitudes of the stepsizes of AdGD-1 and AdGD-2.

\section{Extension of the stepsize Interval Guaranteeing Global Convergence}
In this section, based on Adaptive Gradient Descent, we consider the range of stepsizes that ensure the global convergence of the sequence $\{x^k\}$, and expand it.

\subsection{stepsize Interval Guaranteeing Global Convergence}
In this section, we consider the stepsize interval that ensures convergence to a solution in Adaptive Gradient Descent.

First, let us consider the mathematical meanings of the upper and lower bounds in the stepsize interval that ensure global convergence.
Without an upper bound, the stepsize can become arbitrarily large, in which case the change in the sequence of iterations $\{x^k\}$ per step becomes too large, potentially causing it to diverge without converging.
Conversely, without a lower bound, the stepsize can become infinitesimally small, in which case the progress of the sequence $\{x^k\}$ at each iteration becomes negligible.
This delays convergence and, in the worst case, prevents the sequence from converging to a solution.
Therefore, the upper bound exists to prevent the sequence $\{x^k\}$ from diverging, while the lower bound exists to prevent the sequence $\{x^k\}$ from failing to reach a solution.

Then we specifically consider how the upper and lower bounds should be configured.
Let $\alpha_k$ and $\beta_k$ be the stepsizes of AdGD-1 and AdGD-2.
As can be seen in \cite{malitsky2024}, the convergence of AdGD-1 is proved using the following two conditions:
\begin{align*}
    &\alpha_k\leq\sqrt{1+\theta_{k-1}}\alpha_{k-1},\\
    &\alpha_k\leq\frac{1}{\sqrt{2}L_k}.
\end{align*}
Similarly, for AdGD-2, its convergence is proved using the following two conditions:
\begin{align*}
    &\beta_k\leq\sqrt{\frac{2}{3}+\theta_{k-1}}\beta_{k-1},\\
    &\beta_k\leq\frac{\beta_{k-1}}{\sqrt{[2\beta_{k-1}^2L_k^{2}-1]_+}}.
\end{align*}
Therefore, if we denote the upper bounds of the stepsize intervals that ensure the convergence of AdGD-1 and AdGD-2 as $u_k^{(1)}$ and $u_k^{(2)}$, they are expressed as follows:
\begin{align*}
&u_k^{(1)}=\min\left\{\sqrt{1+\theta_{k-1}}\alpha_{k-1},\quad\frac{1}{\sqrt{2}L_k}\right\},\\
&u_k^{(2)}=\min\left\{\sqrt{\frac{2}{3}+\theta_{k-1}}\beta_{k-1},\quad\frac{\beta_{k-1}}{\sqrt{[2\beta_{k-1}^2L_k^2-1]_+}}\right\}.
\end{align*}

On the other hand, let us consider the lower bound.
In the proof of convergence for Adaptive Gradient Descent, the property that "there exists a strictly positive constant $\tau>0$ satisfying $t_k\geq\tau$ for arbitrary stepsize $t_k$" is used \cite{malitsky2024}.
Therefore, the lower bound of the stepsize that guarantees convergence only needs to be bounded below, and its infimum must be strictly greater than $0$.
Thus, for $\delta\in(0,1)$, we set the lower bounds $l_k^{(1)}$ and $l_k^{(2)}$ of the stepsize intervals for AdGD-1 and AdGD-2 as follows:
\begin{align*}
&l_k^{(1)}=\min\left\{\sqrt{1+\theta_{k-1}}\alpha_{k-1},\quad\frac{\delta}{\sqrt{2}L_k}\right\},\\
&l_k^{(2)}=\min\left\{\sqrt{\frac{2}{3}+\theta_{k-1}}\beta_{k-1},\quad\frac{\delta}{\sqrt{2}L_k}\right\}.
\end{align*}
At this time, since $l_k^{(1)}$ and $l_k^{(2)}$ are bounded below and their infimums are strictly greater than $0$ (we will show in Sections 3.2 and 3.3), convergence is indeed ensured.

In other words, if the stepsize is chosen such that $t_k\in[l_k^{(1)},u_k^{(1)}]$ or $t_k\in[l_k^{(2)},u_k^{(2)}]$, the gradient descent method converges to a solution (since clearly $l_k^{(1)}\leq u_k^{(1)}$, and $l_k^{(2)}\leq u_k^{(2)}$ holds because $\frac{\delta}{\sqrt{2}L_k}\leq\frac{1}{\sqrt{2}L_k}\leq\frac{\beta_{k-1}}{\sqrt{[2\beta_{k-1}^2L_k^2-1]_+}}$, these closed intervals certainly exist).

\subsection{Extension of the stepsize Upper Bound}
In this section, we extend the upper bound of the stepsize interval that guarantees the global convergence of AdGD-1 and AdGD-2.

Let the stepsize $t_k$ be given as follows.
\begin{mybox}{Extension of the stepsize upper bound for AdGD-1}
\begin{align}
    &l_k^{(1)}=\min\left\{\sqrt{1+\theta_{k-1}}t_{k-1},\quad\frac{\delta}{\sqrt{2}L_k}\right\}\label{swadgdstep1}\\
    &u_k^{(1)}=\min\left\{\sqrt{1+\theta_{k-1}}t_{k-1},\quad\frac{1}{L_k}\sqrt{\frac{1}{2}\left(1+\frac{\gamma_k}{\|x^k-x^{k-1}\|^2}\right)}\right\}\label{swadgdstep2}\\
    &t_k\in\left[l_k^{(1)},u_k^{(1)}\right]\label{swadgdstep}
\end{align}
  where $L_k$ is the same as in equation \eqref{locallipschiz} and $\theta_k$ is the same as in equation \eqref{theta}.
  In addition, $\delta\in(0,1)$, and $\{\gamma_k\}$ is a sequence of real numbers satisfying $\gamma_k\geq0$ and $\sum_{k=1}^{\infty}\gamma_k<\infty$.
\end{mybox}

Under these conditions, the sequence of iterations $\{x^k\}$ generated by gradient descent has an accumulation point, and this accumulation point is a solution.
We can see that the upper bound is larger than that of the stepsize interval shown in the previous section by the term involving $\gamma_k$.

We prove this fact following a line of reasoning similar to the proof of the global convergence of AdGD-1.

First, we show that an inequality containing $\|x^k-x^*\|^2$ holds, from which a telescoping sum with respect to $k$ can be taken.

\begin{lemma}\label{neq1}
If the stepsize $t_k$ satisfies \eqref{swadgdstep}, the sequence $\{x^k\}$, a solution $x^*$, and the optimal value $f_*$ satisfy the following inequality;
\begin{equation}\label{neq3}
  \begin{split}
    \|x^{k+1}-x^*\|^2+&\|x^{k+1}-x^k\|^2+2t_k(1+\theta_k)(f(x^k)-f_*)\\
    &\leq\|x^k-x^*\|^2+\|x^k-x^{k-1}\|^2+2t_k\theta_k(f(x^{k-1})-f_*)+\gamma_k.
  \end{split}
\end{equation}
\end{lemma}
\begin{proof}
Rewriting $\|x^{k+1}-x^k\|^2$, we have
\begin{align}
  \|x^{k+1}-x^k\|^2&=t_k^2\|\nabla f(x^k)\|^2\nonumber\\
  &=t_k^2\left(\|\nabla f(x^k)-\nabla f(x^{k-1})\|^2-\|\nabla f(x^{k-1})\|^2+2\left\langle\nabla f(x^k),\nabla f(x^{k-1})\right\rangle\right)\nonumber\\
  &=t_k^2L_k^2\|x^k-x^{k-1}\|^2-t_k^2\|\nabla f(x^{k-1})\|^2+2t_k^2\left\langle\nabla f(x^k),\nabla f(x^{k-1})\right\rangle.\label{swadgdneq1}
\end{align}
From \eqref{swadgdstep2} and \eqref{swadgdstep}, we obtain
\[
t_k\leq\frac{1}{L_k}\sqrt{\frac{1}{2}\left(1+\frac{\gamma_k}{\|x^k-x^{k-1}\|^2}\right)}.
\]
Squaring both sides and rearranging, we obtain
\begin{equation*}
  t_k^2L_k^2\leq\frac{1}{2}+\frac{\gamma_k}{2\|x^k-x^{k-1}\|^2}.
\end{equation*}
Using this and Lemma \ref{norm1}, equation \eqref{swadgdneq1} becomes
\begin{align*}
    \|x^{k+1}-x^k\|^2&\leq\frac{1}{2}\|x^k-x^{k-1}\|^2+\frac{\gamma_k}{2}-t_k^2\|\nabla f(x^{k-1})\|^2+2t_k^2\left\langle\nabla f(x^k),\nabla f(x^{k-1})\right\rangle\\
    &\leq\frac{1}{2}\|x^k-x^{k-1}\|^2+t_k^2\left\langle\nabla f(x^k),\nabla f(x^{k-1})\right\rangle+\frac{\gamma_k}{2}\\
    &=\frac{1}{2}\|x^k-x^{k-1}\|^2+t_k\theta_k\left\langle\nabla f(x^k),x^{k-1}-x^k\right\rangle+\frac{\gamma_k}{2}\\
    &\leq\frac{1}{2}\|x^k-x^{k-1}\|^2+t_k\theta_k\left(f(x^{k-1})-f(x^k)\right)+\frac{\gamma_k}{2},
\end{align*}
where the last inequality uses the fact that Lemma \ref{convex1} in the appendix holds for $x^k$ and $x^*$ due to the convexity of $f$.
Multiplying this by 2 and rearranging terms yields
\begin{equation}\label{neq5}
  \|x^{k+1}-x^k\|^2\leq\|x^k-x^{k-1}\|^2-\|x^{k+1}-x^k\|^2+2t_k\theta_k(f(x^{k-1})-f(x^k))+\gamma_k.
\end{equation}
By adding this to equation \eqref{standard}, we obtain
\begin{equation*}
  \begin{split}
    \|x^{k+1}-x^*\|^2+\|x^{k+1}-x^k\|^2\leq&\|x^k-x^*\|^2-2t_k(f(x^k)-f_*)+t_k^2\|\nabla f(x^k)\|^2\\
    &+\|x^k-x^{k-1}\|^2-\|x^{k+1}-x^k\|^2+2t_k\theta_k(f(x^{k-1})-f(x^k))+\gamma_k.
  \end{split}
\end{equation*}
Rearranging this by using equation \eqref{gd}, the following holds;
\begin{equation*}
  \begin{split}
    \|x^{k+1}-x^*\|^2+&\|x^{k+1}-x^k\|^2+2t_k(1+\theta_k)(f(x^k)-f_*)\\
    &\leq\|x^k-x^*\|^2+\|x^k-x^{k-1}\|^2+2t_k\theta_k(f(x^{k-1})-f_*)+\gamma_k.
  \end{split}
\end{equation*}
\end{proof}

Then, we show that the sequence $\{x^k\}$ is contained within a closed ball centered on a solution $x^*$ by taking the telescoping sum of the inequality obtained in Lemma \ref{neq1}.
\begin{lemma}\label{swadgdbounds}
  If the stepsize $t_k$ satisfies \eqref{swadgdstep}, the sequence of iterations $\{x^k\}$ is bounded.
  Furthermore, for a solution $x^*$, $x^k\in B(x^*,R)$ holds, where $R$ is a constant determined by $x^*$, the initial values, and $\{\gamma_k\}$.
\end{lemma}
\begin{proof}
Taking the telescoping sum of \eqref{neq3}, we get
\begin{align}
  \|x^{k+1}-x^*\|^2&+\|x^{k+1}-x^k\|^2+2\sum_{i=1}^{k}b_i(f(x^i)-f_*)\nonumber\\
  &\leq\|x^1-x^*\|^2+\|x^1-x^0\|^2+2t_1\theta_1(f(x^0)-f_*)+\sum_{i=1}^k\gamma_i,\label{swadgdneq2}
\end{align}
where $\{b_i\}$ is a sequence of real numbers defined as follows:
\begin{equation}\label{b}
  b_i=
  \begin{cases}
    t_i(1+\theta_i)                     & \text{if $i=k$,} \\
    t_i(1+\theta_i)-t_{i+1}\theta_{i+1} & \text{if $1\leq i\leq k-1$.} 
  \end{cases}
\end{equation}
Also, from \eqref{swadgdstep2} and \eqref{swadgdstep}, we have $t_k\leq\sqrt{1+\theta_{k-1}}t_{k-1}$.
Squaring both sides yields
\begin{equation*}
  t_k^2\leq(1+\theta_{k-1})t_{k-1}^2,
\end{equation*}
and dividing this by $t_{k-1}$ gives
\begin{equation}\label{swadgdstep3}
  t_k\theta_k\leq(1+\theta_{k-1})t_{k-1}.
\end{equation}

Here, from \eqref{gd}, \eqref{standard}, and \eqref{swadgdstep3}, we have
\begin{align*}
    \|x^1-x^*\|^2&+\|x^1-x^0\|^2+2t_1\theta_1(f(x^0)-f_*)\\
    &\leq\|x^0-x^*\|^2+2(t_1\theta_1-t_0)(f(x^0)-f_*)+2t_0^2\|\nabla f(x^0)\|^2\\
    &\leq\|x^0-x^*\|^2+2t_0\theta_0(f(x^0)-f_*)+2t_0^2\|\nabla f(x^0)\|^2.
\end{align*}
Furthermore, since $\sum_{k=1}^{\infty}\gamma_k$ is bounded by assumption, using a certain real number $S>0$, the following inequality holds for arbitrary $k\in\mathbb{N}$;
\begin{equation*}
  \sum_{i=1}^k\gamma_i\leq S.
\end{equation*}
Therefore, from \eqref{swadgdneq2},
\begin{align}\label{swadgdupper}
  \|x^{k+1}-x^*\|^2+&\|x^{k+1}-x^k\|^2+2\sum_{i=1}^{k}b_i(f(x^i)-f_*) \nonumber\\
  &\leq\|x^0-x^*\|^2+2t_0\theta_0(f(x^0)-f_*)+2t_0^2\|\nabla f(x^0)\|^2+S \nonumber\\
  &\colon=R^2
\end{align}
holds (where we assume $R>0$).

Here,
\begin{itemize}
\item When $i=k$, $b_k=t_k(1+\theta_k)\geq0$ since $t_k\geq0$ and $(1+\theta_k)>0$.
\item When $1\leq i\leq k-1$, $b_i=t_i(1+\theta_i)-t_{i+1}\theta_{i+1}\geq0$ from equation \eqref{swadgdstep3}.
\end{itemize}
Thus, $b_i\geq0$ for $1\leq i\leq k$. Furthermore, since $f_*$ is the optimal value, $f(x^i)-f_*\geq0$ holds.

Therefore, the following inequality holds;
\begin{equation*}
  \|x^{k+1}-x^*\|^2\leq R^2.
\end{equation*}
Since this holds for any $k$, $\{x^k\}$ is bounded and $x^k\in B(x^*,R)$ holds.
\end{proof}

From this lemma, even if $f$ is locally smooth, by restricting our discussion within this closed ball, we can definitively treat it as smooth with an existing Lipschitz constant $L$.

As described in Section 3.1, the proof of convergence also requires that the stepsize is bounded below by a strictly positive constant, which can be shown using the Lipschitz constant $L$.
The next lemma demonstrates this fact.
\begin{lemma}\label{positive}
If $t_k$ satisfies \eqref{swadgdstep}, it satisfies the following inequality;
\begin{equation}\label{minstep-1}
  t_k\geq\min\left\{t_0,\frac{\delta}{\sqrt{2}L}\right\},
\end{equation}
where $L$ is the Lipschitz constant of $\nabla f(x)$ within the closed ball $B(x^*,R)$ from Lemma \ref{swadgdbounds}.
\end{lemma}
\begin{proof}
  First, depending on the value of $l_k^{(1)}$, the possible interval for $t_k$ can be categorized into the following cases:
  \begin{description}
    \item[(Case 1)] When $l_k^{(1)}=\sqrt{1+\theta_{k-1}}t_{k-1}$, we have $\sqrt{1+\theta_{k-1}}t_{k-1}\leq\frac{\delta}{\sqrt{2}L_k}$ from \eqref{swadgdstep1}.
    Here, since 
    \[
    \frac{\delta}{\sqrt{2}L_k}<\frac{1}{\sqrt{2}L_k}\leq\frac{1}{L_k}\sqrt{\frac{1}{2}\left(1+\frac{\gamma_k}{\|x^k-x^{k-1}\|^2}\right)}
    \]
    holds, we consequently have $u_k^{(1)}=\sqrt{1+\theta_{k-1}}t_{k-1}$.
    That is, since $l_k^{(1)}=u_k^{(1)}$, we obtain
    \begin{equation}\label{condition1-1}
    t_k=\sqrt{1+\theta_{k-1}}t_{k-1}.
    \end{equation}
    \item[(Case 2)] When $l_k^{(1)}=\frac{\delta}{\sqrt{2}L_k}$, we have $\frac{\delta}{\sqrt{2}L_k}\leq\sqrt{1+\theta_{k-1}}t_{k-1}$ from \eqref{swadgdstep1}.
    Since 
    \[
    \frac{\delta}{\sqrt{2}L_k}<\frac{1}{L_k}\sqrt{\frac{1}{2}\left(1+\frac{\gamma_k}{\|x^k-x^{k-1}\|^2}\right)}
    \]
    also holds, we consequently obtain
    \begin{equation}\label{condition1-2}
      t_k\in\left[\frac{\delta}{\sqrt{2}L_k},u_k^{(1)}\right].
    \end{equation}
  \end{description}

  We prove the proposition using mathematical induction.
  \begin{itemize}
    \item When $k=1$, if $t_1$ falls into (Case 1), $t_1=\sqrt{1+\theta_0}t_0\geq t_0$ from \eqref{condition1-1}. If $t_1$ falls into (Case 2), $t_1\geq\frac{\delta}{\sqrt{2}L_k}\geq\frac{\delta}{\sqrt{2}L}$ from \eqref{condition1-2}. Therefore, \eqref{minstep-1} holds in both cases.
    \item Assume that $t_{k-1}$ satisfies \eqref{minstep-1}. First, if $t_k$ falls into (Case 1), $t_k=\sqrt{1+\theta_{k-1}}t_{k-1}\geq t_{k-1}$ from \eqref{condition1-1}; thus, $t_k$ also satisfies \eqref{minstep-1} by the induction hypothesis. On the other hand, if $t_k$ falls into (Case 2), $t_k\geq\frac{\delta}{\sqrt{2}L_k}\geq\frac{\delta}{\sqrt{2}L}$ from \eqref{condition1-2}, which satisfies \eqref{minstep-1}. Therefore, $t_k$ satisfies \eqref{minstep-1} in both cases.
  \end{itemize}
  Consequently, \eqref{minstep-1} holds for any $k\in\mathbb{N}$.
\end{proof}

From these lemmas, the following theorem regarding convergence to a solution holds.
\begin{theorem}\label{swadgd}
Let $f\colon\mathbb{R}^d\rightarrow\mathbb{R}$ be convex with locally Lipschitz gradient $\nabla f$.
If the stepsize $t_k$ satisfies \eqref{swadgdstep}, the sequence $\{x^k\}$ is bounded, and its accumulation point is an optimal solution, where the initial values are $t_0>0$, $\theta_0\geq0$, and $x^0\in\mathbb{R}^d$.
\end{theorem}
\begin{proof}
  From Lemma \ref{swadgdbounds}, $\{x^k\}$ is bounded.
  Thus, an accumulation point exists.

  Here, since $B(x^*,R)$ is a convex set and $f$ is locally smooth, the mapping $f(x)$ with domain $x\in B(x^*,R)$ is L-smooth for a certain constant $L>0$.
  Furthermore, since $x^*,x^k\in B(x^*,R)$ from Lemma \ref{swadgdbounds}, Lemma \ref{lipschitz2} in Appendix yields
  \begin{equation*}
    f(x^*)-f(x^k)\geq\langle\nabla f(x^k),x^*-x^k\rangle+\frac{1}{2L}\|\nabla f(x^k)\|^2.
  \end{equation*}
  Using this in \eqref{standard} instead of the inequality from the convexity of $f$, the following inequality holds;
  \begin{equation}\label{standard2}
    \|x^{k+1}-x^*\|^2\leq\|x^k-x^*\|^2-2t_k\left(f(x^k)-f(x^*)\right)-\frac{t_k}{L}\|\nabla f(x^k)\|^2+t_k^2\|\nabla f(x^k)\|^2.
  \end{equation}
  Adding this to equation \eqref{neq5} yields
  \begin{equation*}
    \begin{split}
      \|x^{k+1}-x^*\|^2+&\|x^{k+1}-x^k\|^2+2t_k(1+\theta_k)(f(x^k)-f_*)+\frac{t_k}{L}\|\nabla f(x^k)\|^2\\
      &\leq\|x^k-x^*\|^2+\|x^k-x^{k-1}\|^2+2t_k\theta_k(f(x^{k-1})-f_*)+\gamma_k.
    \end{split}
  \end{equation*}
  Taking the telescoping sum in the same manner as Lemma \ref{swadgdbounds}, the following inequality holds;
  \begin{equation*}
    \|x^{k+1}-x^*\|^2+\|x^{k+1}-x^k\|^2+2\sum_{i=1}^{k-1}b_i(f(x^i)-f_*)+\sum_{i=1}^k\frac{t_i}{L}\|\nabla f(x^i)\|^2\leq R^2,
  \end{equation*}
  where $b_i$ is the same as in \eqref{b}.
  At this time, since $b_i\geq0$ and $f(x^i)-f_*\geq0$ similarly to Lemma \ref{swadgdbounds},
  \begin{equation*}
    \sum_{i=1}^k\frac{t_i}{L}\|\nabla f(x^i)\|^2\leq R^2
  \end{equation*}
  holds.
  Therefore, taking the limit as $k\to\infty$, the following inequality holds;
  \begin{equation*}
    \sum_{i=1}^{\infty}\frac{t_i}{L}\|\nabla f(x^i)\|^2\leq R^2.
  \end{equation*}
  From this, $\lim_{i\to\infty}t_i\|\nabla f(x^i)\|^2=0$ holds.
  At this time, since $t_i\geq\min\left\{t_0,\frac{\delta}{\sqrt{2}L}\right\}>0$ from Lemma \ref{positive},
  \begin{equation*}
    \lim_{i\to\infty}\|\nabla f(x^i)\|=0
  \end{equation*}
  holds.
  
  For any accumulation point $\lim_{j\to\infty}x^{i_j}=\bar{x}$, we have $\lim_{j\to\infty}\|\nabla f(x^{i_j})\|=\|\nabla f(\bar{x})\|=0$, meaning that $\bar{x}$ is an optimal solution.
\end{proof}

Similarly, the stepsize interval for AdGD-2 can also be extended.
Let the stepsize $t_k$ be given as follows.
\begin{mybox}{Extension of the stepsize upper bound for AdGD-2}
  \begin{align}
    &l_k^{(2)}=\min\left\{\sqrt{\frac{2}{3}+\theta_{k-1}}t_{k-1},\quad\frac{\delta}{\sqrt{2}L_k}\right\}\label{wadgdstep1},\\
    &u_k^{(2)}=\min\left\{\sqrt{\frac{2}{3}+\theta_{k-1}}t_{k-1},\quad\sqrt{\frac{1+\frac{\gamma_k}{\|x^{k}-x^{k-1}\|^2}}{[2t_{k-1}^2L_k^2-1]_+}}t_{k-1}\right\}\label{wadgdstep2},\\
    &t_k\in\left[l_k^{(2)},u_k^{(2)}\right]\label{wadgdstep},
  \end{align}
  where $L_k$ is the same as in equation \eqref{locallipschiz} and $\theta_k$ is the same as in equation \eqref{theta}.
  In addition, $\delta\in(0,1)$, and $\{\gamma_k\}$ is a sequence of real numbers satisfying $\gamma_k\geq0$ and $\sum_{k=1}^{\infty}\gamma_k<\infty$.
\end{mybox}

Under these conditions, the following theorem holds.
\begin{theorem}\label{wadgd}
Let $f\colon\mathbb{R}^d\rightarrow\mathbb{R}$ be convex with locally Lipschitz gradient $\nabla f$.
Then, if the stepsize $t_k$ satisfies \eqref{wadgdstep}, the sequence $\{x^k\}$ is bounded, and its accumulation point is a solution.
The initial values are $t_0>0$, $\theta_0>0$, and $x^0\in\mathbb{R}^d$.
\end{theorem}

Since the proof is almost the same as that for AdGD-1, it is provided in the appendix.

\section{Combination with Existing Stepsize Calculation Methods}
In Chapter 3, we derived the stepsize intervals that ensures convergence to a solution in Adaptive Gradient Descent.
Therefore, in this chapter, we consider a way to adopt stepsize calculation methods that are practical but have weak convergence theory, while ensuring their convergence by utilizing these intervals.

In this paper, we specifically propose a stepsize calculation method that utilizes the Barzilai-Borwein method (BB method).
The BB method was proposed by Barzilai and Borwein \cite{BB}.

The BB method is a variant of the quasi-Newton method\footnote{
  Quasi-Newton methods are a class of descent methods, but unlike the gradient descent method, the descent direction $d_k$ is set as $d_k=-B_k^{-1}\nabla f(x^k)$ or $d_k=-H_k\nabla f(x^k)$, where $B_k$ is an approximation matrix of the Hessian matrix $\nabla^2f(x^k)$ and $H_k=B_k^{-1}$.
  }
and generates the sequence $\{x^k\}$ using the following update formula;
\begin{align*}
  x^{k+1}=x^k-t_k B_k^{-1}\nabla f(x^k)\quad(\text{or } x^{k+1}=x^k-t_k H_k\nabla f(x^k)).
\end{align*}
In this case, the stepsize $t_k$ is chosen to minimize the function value at the next step, such as $t_k=\text{argmin}_{t>0}f(x^k+t d_k)$ (this is called a line search). Furthermore, the approximation matrix $B_k$ (or $H_k$) is given by $B_k=\lambda_k^{(1)}I$ ($H_k=\lambda_k^{(2)}I$), where
\begin{align*}
  \lambda_k^{(1)}&=\frac{(x^k-x^{k-1})^T(\nabla f(x^k)-\nabla f(x^{k-1}))}{\|x^k-x^{k-1}\|^2}\\
  \lambda_k^{(2)}&=\frac{(x^k-x^{k-1})^T(\nabla f(x^k)-\nabla f(x^{k-1}))}{\|\nabla f(x^k)-\nabla f(x^{k-1})\|^2}
\end{align*}

Here, if we consider the case without line search—that is, assuming the stepsize $t_k$ is absent—the BB method can be viewed as a gradient descent method with a stepsize of $\frac{1}{\lambda_k^{(1)}}$ (or $\lambda_k^{(2)}$).
For this BB method without line search, superlinear convergence has been proven under very strong assumptions if the objective function is a strictly convex quadratic function \cite{BB2}.
Therefore, although the convergence rate is fast, the class of objective functions for which convergence can be guaranteed is inherently very limited.

To resolve this issue, we consider incorporating the BB method into Adaptive Gradient Descent.
Specifically, using the stepsize interval that ensures global convergence considered in Chapter 3, we adopt BB stepsize $t_{BB}$ if it lies within that interval; otherwise, we adopt either the upper or lower bound.
By doing so, as proven in Chapter 3, global convergence can be guaranteed for objective functions that are locally smooth.

In this paper, we employ $\lambda_k^{(2)}$ as $t_{BB}$, which is the stepsize of the Barzilai-Borwein method.
Using this, we propose a new gradient descent method in which the stepsize is determined as follows.
\begin{mybox}{Extended BB\_AdGD-1}
\begin{align}
  &t_{BB}=\frac{(x^k-x^{k-1})^T(\nabla f(x^k)-\nabla f(x^{k-1}))}{\|\nabla f(x^k)-\nabla f(x^{k-1})\|^2}\label{bbstep},\\
  &t_k=\text{mid}\left\{t_{BB},l_k^{(1)},u_k^{(1)}\right\}\label{bbadgd1},
\end{align}
where $l_k^{(1)}$ and $u_k^{(1)}$ are the same as defined in equations \eqref{swadgdstep1} and \eqref{swadgdstep2}.
In addition, $\text{mid}\{\cdot\}$ denotes the median.
\end{mybox}
We refer to this method as Extended BB\_AdGD-1.

At this time, the value taken by $t_k$ can be classified into the following three cases depending on the order of $t_{BB}$, $l_k^{(1)}$ and $u_k^{(1)}$.
\begin{enumerate}
  \item In the case of $t_{BB}\leq l_k^{(1)}\leq u_k^{(1)}$, we have $t_k=l_k^{(1)}$, it clearly holds that $t_k\in[l_k^{(1)},u_k^{(1)}]$.
  \item In the case of $l_k^{(1)}\leq t_{BB}\leq u_k^{(1)}$, we have $t_k=t_{BB}$, which clearly implies $t_k\in[l_k^{(1)},u_k^{(1)}]$ under this assumption.
  \item In the case of $l_k^{(1)}\leq u_k^{(1)}\leq t_{BB}$, we have $t_k=u_k^{(1)}$, it clearly holds that $t_k\in[l_k^{(1)},u_k^{(1)}]$.
\end{enumerate}

Therefore, $t_k\in[l_k^{(1)},u_k^{(1)}]$ holds. Consequently, from Theorem \ref{swadgd}, assuming only the convexity and local L-smoothness of the objective function, the accumulation point of the sequence of iterations $\{x^k\}$ generated by Extended BB\_AdGD-1 is an optimal solution.
That is, unlike the standard BB method, it can be seen that convergence can be guaranteed as long as the function is locally smooth, even if it is not globally smooth.

Similarly, BB method can be combined with AdGD-2.
\begin{mybox}{Extended BB\_AdGD-2}
\begin{equation}\label{bbadgd2}
  t_k=\text{mid}\left\{t_{BB},l_k^{(2)},u_k^{(2)}\right\}
\end{equation}
where $l_k^{(2)}$ and $u_k^{(2)}$ are the same as defined by \eqref{wadgdstep1} and \eqref{wadgdstep2}, and $t_{BB}$ is the same as defined in equation \eqref{bbstep}.
In addition, $\text{mid}$ denotes the median.
\end{mybox}
We refer to this method as Extended BB\_AdGD-2.

Reasoning in the same manner as for Extended BB\_AdGD-1 yields $t_k\in[l_k^{(2)},u_k^{(2)}]$.
Therefore, from Theorem \ref{wadgd}, assuming only that the objective function is convex and locally smooth, the accumulation point of the sequence of iterations $\{x^k\}$ generated by Extended BB\_AdGD-2 is an optimal solution.

\section{Numerical Experiments}

In this section, we conduct numerical experiments to compare the performance of the proposed methods with the original adaptive gradient descent methods.
The numerical experiments were performed on a computer with macOS Sequoia 15.6, an Apple M3 CPU, and 24 GB of memory, and the algorithms were implemented using Python 3.14.1.

Specifically, we compare the following six methods:
\begin{itemize}
  \item AdGD-1
  \item AdGD-2
  \item A method that adopts the stepsize of Extended BB\_AdGD-1 \eqref{bbadgd1} with $\gamma_k=0$ (denoted as BB\_AdGD-1)
  \item A method that adopts the stepsize of Extended BB\_AdGD-2 \eqref{bbadgd2} with $\gamma_k=0$ (denoted as BB\_AdGD-2)
  \item Extended BB\_AdGD-1 (denoted as BB\_exAdGD-1)
  \item Extended BB\_AdGD-2 (denoted as BB\_exAdGD-2)
\end{itemize}
Here, BB\_AdGD-1 and BB\_AdGD-2 are methods that combine the stepsize interval before extending the upper bound with the BB method in the same manner as in Section 4.

The comparison is focused on the following three aspects:
\begin{itemize}
  \item Simple comparison of the six methods.
  \item Comparison based on the sequence $\gamma_k$ that extends the upper bound.
  \item Comparison based on the parameter $\delta$ that sets the lower bound.
\end{itemize}

\begin{description}
  \item[1. Simple Comparison of Methods] 
  We compare the six methods mentioned above by setting the sequence $\{\gamma_k\}$ as $\displaystyle\gamma_k=\frac{\|x^1-x^0\|}{k^2}$ and fixing the parameter $\delta$ to $\delta=0.5$.
  \item[2. Comparison of the Sequence $\{\gamma_k\}$] 
  For Extended BB\_AdGD-1, we compare the 4 cases about the sequence $\{\gamma_k\}$; 
  \begin{equation*}
    \gamma_k^1=
    \begin{cases}
      \|x^1-x^0\|                    & \text{if $k=1$,} \\
      \frac{\|x^1-x^0\|}{k(\log k)^2} & \text{if $k\geq2$,} 
    \end{cases}
  \end{equation*}
  $\displaystyle\gamma_k^2=\frac{\|x^1-x^0\|}{k^2}$, $\displaystyle\gamma_k^3=\frac{\|x^1-x^0\|}{k^3}$, and $\displaystyle\gamma_k^4=\|x^1-x^0\|\left(\frac{3}{4}\right)^{k-1}$.

  (Parameter $\delta=0.5$.)
  \item[3. Comparison of the Parameter $\delta$] 
  For Extended BB\_AdGD-1, we compare the results by varying the parameter $\delta$ from $0.1$ to $1$ in increments of $0.05$.
  At this time, the sequence $\gamma_k$ is $\displaystyle\gamma_k=\frac{\|x^1-x^0\|}{k^2}$.
\end{description}

In these settings, $\gamma_k>0$ and $\sum_{i=1}^{\infty}\gamma_k<\infty$ \cite{mathar2009} hold.

Regarding the initial values, the initial stepsize is set to $t_0=1$.
For the initial value of $\theta$, we set $\theta_0=0$ for AdGD-1, BB\_AdGD-1 and Extended BB\_AdGD-1, and $\theta_0=\frac{1}{3}$ for the other methods.

The stopping criterion for each algorithm is that the relative error becomes smaller than $10^{-10}$, i.e., $\frac{\|\nabla f(x^k)\|}{\|\nabla f(x^0)\|}<10^{-10}$.

\paragraph{Logistic Regression}
As the objective function, we consider the logistic loss with $l_2$-regularization;
\begin{equation*}
  f(x)=-\frac{1}{n}\sum_{i=1}^n\left(y_i\log(s(a_i^Tx))+(1-y_i)\log(1-s(a_i^Tx))\right)+\frac{\lambda}{2}\|x\|^2,
\end{equation*}
where $x\in\mathbb{R}^d$, $a_i\in\mathbb{R}^d$, $y_i\in\{0,1\}$, and $s(z)$ is the sigmoid function $s(z)=\frac{1}{1+\exp(-z)}$.
The presence of the regularization term $\frac{\lambda}{2}\|x\|^2$ ensures that $f$ is a strongly convex function when $\lambda>0$.
In this experiments, we set $\lambda$ as $\displaystyle\lambda=\frac{L}{n}$, where $L$ is Lipshitz constant, $L =\frac{\lambda_{\max}(A^TA)}{4n}$ ($A=(a_{ij})$).

In this experiment, we used the "coil2000" dataset from OpenML for $a_i$ and $y_i$.
The dimension of the data is $d=86$ and the number of data points is $n=9822$.
The initial value of the iterations was set to $x^0=(1,1,\cdots,1)$.

\paragraph{Experimental Results}
The results are shown in Figure \ref{logerror1}, \ref{logerror2} and Table \ref{logtable}.

\begin{figure}[htbp]
\centering
\includegraphics[width=0.5\columnwidth]{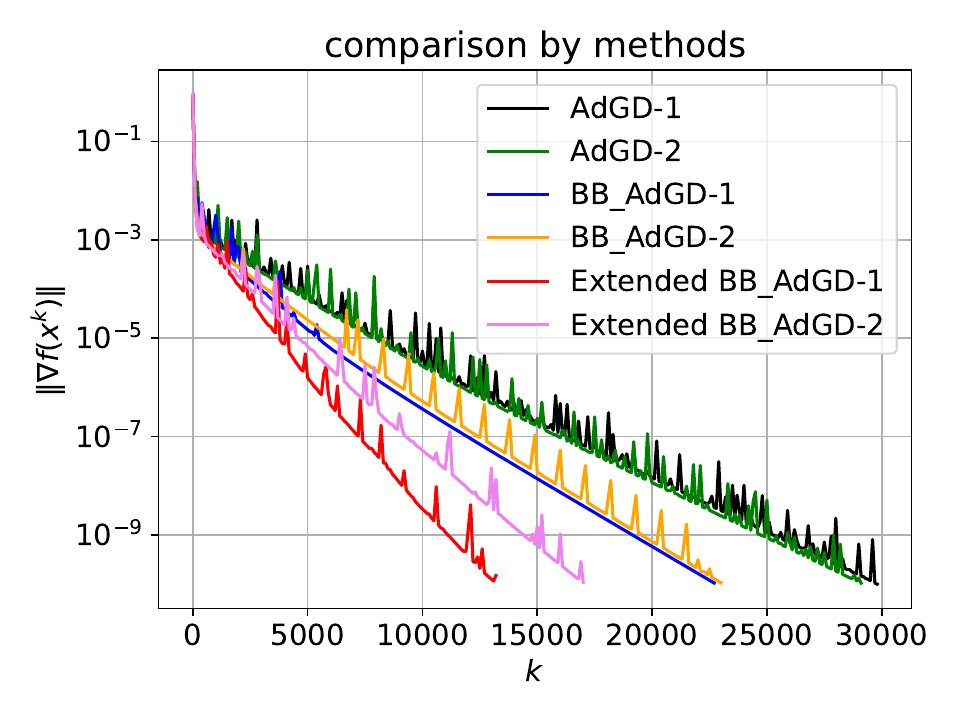}
\caption{Comparison by methods. Horizontal axis means the number of steps and vertical axis means the relative error $\frac{\|\nabla f(x^k)\|}{\|\nabla f(x^0)\|}$.}
\label{logerror1}
\end{figure}

\begin{figure}[htbp]
\centering
\begin{minipage}[b]{0.49\columnwidth}
    \centering
    \includegraphics[width=0.9\columnwidth]{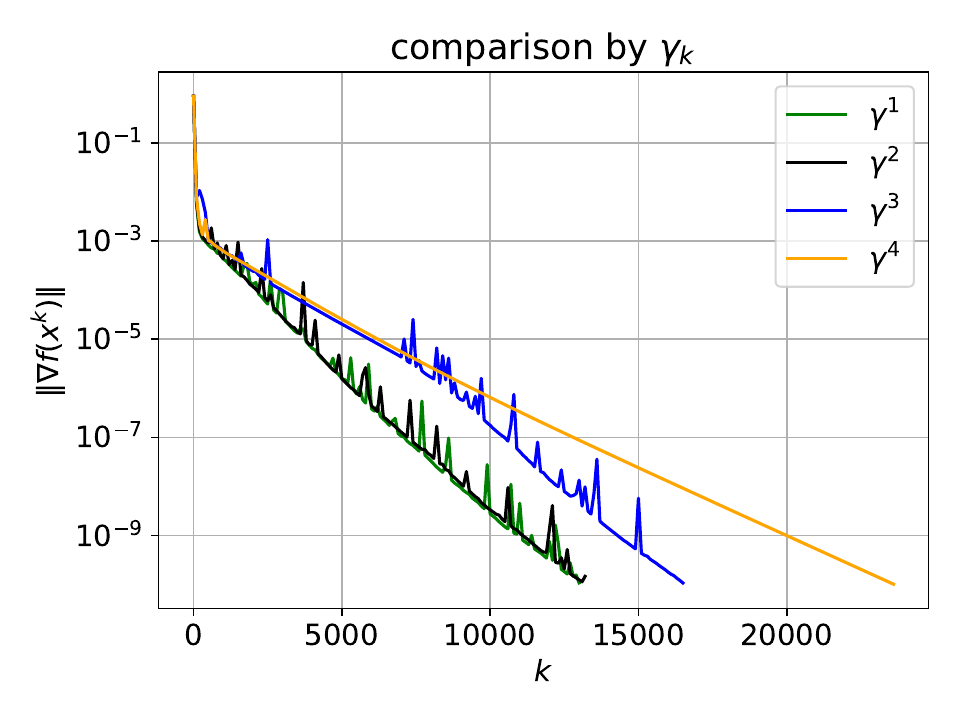}
\end{minipage}
\begin{minipage}[b]{0.49\columnwidth}
    \centering
    \includegraphics[width=0.9\columnwidth]{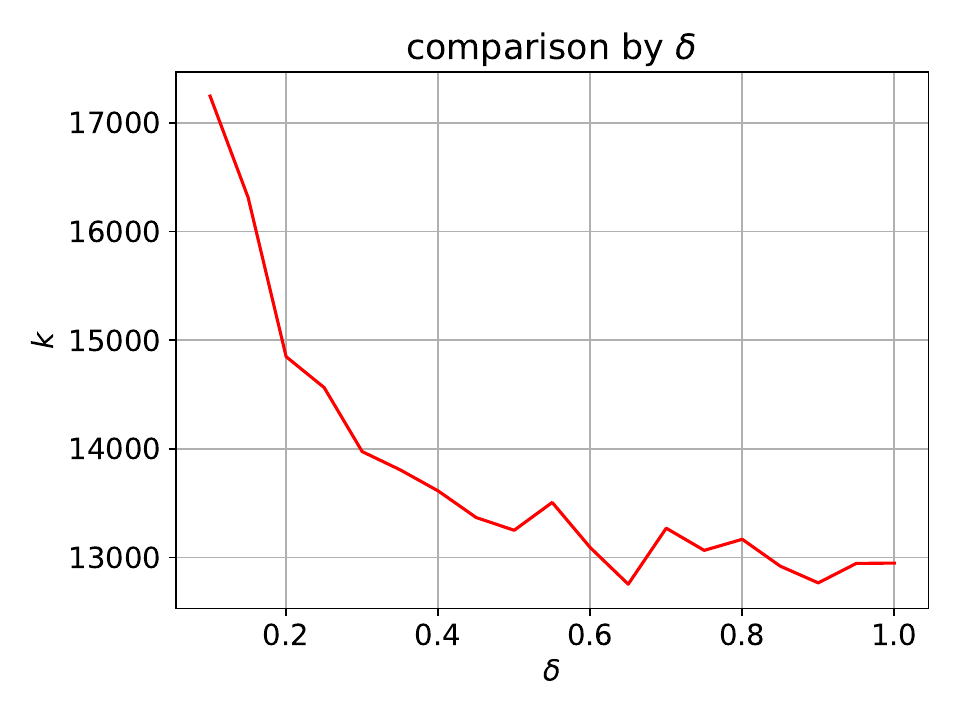}
\end{minipage}
\caption{Left shows comparison by $\gamma_k$, and right shows comparison by $\delta$. Horizontal axis means the number of steps and vertical axis means the relative error $\frac{\|\nabla f(x^k)\|}{\|\nabla f(x^0)\|}$.}
\label{logerror2}
\end{figure}

\begin{table}[h]
  \centering
  \caption{The number of steps until completion, the proportion of each type of stepsize adopted (Case 1 to Case 4), and the computation time until completion (s) for each method.
  Case 1 represents the case where $t_k=l_k=u_k$, Case 2 represents $t_k=l_k$ excluding Case 1, Case 3 represents $t_k=t_{BB}$ excluding Case 1, and Case 4 represents $t_k=u_k$ excluding Case 1.}
  \label{logtable}
  \begin{tabular}{|l|r|r|r|r|r|r|} \hline
    Method & Total Steps & Case.1 & Case.2 & Case.3 & Case.4 & Time (s)\\ \hline
    AdGD-1 & 29800 & - & - & - & - & 6.81 \\
    AdGD-2 & 29189 & - & - & - & - & 6.77 \\
    BB\_AdGD-1 & 22779 & 0.674 & 0.108 & 0.091 & 0.125 & 5.38 \\
    BB\_AdGD-2 & 23082 & 0.682 & 0.096 & 0.046 & 0.175 & 5.41 \\
    Extended BB\_AdGD-1 & 13250 & 0.620 & 0.140 & 0.219 & 0.021 & 3.10 \\
    Extended BB\_AdGD-2 & 17056 & 0.697 & 0.107 & 0.142 & 0.054 & 4.33 \\ \hline
  \end{tabular}
\end{table}

From the results, it can be seen that Extended BB\_AdGD-1 requires the fewest number of steps to converge and achieves the shortest computation time.
Furthermore, it is observed that Extended BB\_AdGD-1 has the highest proportion of adopting the BB stepsize.
This means Extended BB\_AdGD-1 and Extended BB\_AdGD-2 are more likely to adopt the BB stepsize than BB\_AdGD-1 and BB\_AdGD-2.

Moreover, it can be seen that the larger the sequence $\gamma_k$, the smaller the number of iterations required for convergence in Extended BB\_AdGD-1.
Regarding the parameter $\delta$, it can be observed that an excessively small value increases the number of iterations required for convergence, while sufficiently large values yield almost no change in the iteration count.
Since a higher proportion of adopting the BB method corresponds to a smaller number of steps to convergence, these results suggest that extending the stepsize upper bound is indeed meaningful and effective.

\section{Conclusion and Perspective}
In this paper, we derived a stepsize interval that guarantees the convergence of Adaptive Gradient Descent to a solution, and demonstrated that its upper bound can be further extended.
By utilizing this extended interval, we proposed a novel method that integrates heuristic approaches—such as BB method, which yields promising numerical results despite lacking theoretical convergence guarantees.
To validate the effectiveness of our proposed method, we conducted numerical experiments using logistic loss as objective functions.
Through these experiments, we compared the convergence rates against the original Adaptive Gradient Descent and evaluated the impact of the extended upper bound.
The results indicate that incorporating the heuristic approach accelerates the convergence speed compared to using Adaptive Gradient Descent alone.

For future work, while this study assumed the objective function to be a locally smooth convex function, general functions are often non-convex or non-differentiable. 
Therefore, we aim to investigate whether our approach can be extended to handle such objective functions.
Additionally, although the Euclidean norm was employed throughout this paper, exploring whether the algorithm can be accelerated by adopting other norms remains an important avenue for future research.

\appendix

\section{Basic Mathematical Properties Used in This Paper}
This section outlines the mathematical properties regarding convex functions and Lipschitz continuity utilized in this paper.

\begin{lemma}((3.2) in \cite{textbook})\label{convex1}
Let $f\colon\mathbb{R}^n\to\mathbb{R}$ be a function whose effective domain $\text{dom}f$ is an open convex set.
When $f$ is differentiable on $\text{dom}f$, a necessary and sufficient condition for $f$ to be convex is that the following inequality holds for any $x,y\in\text{dom}f$ such that $x\neq y$:
\begin{equation*}
f(y)-f(x)\geq\langle\nabla f(x),y-x\rangle.
\end{equation*}\qed
\end{lemma}

\begin{lemma}(Lemma 1 in \cite{malitsky2024})\label{norm1}
Let $\{x^k\}$ be the sequence generated by applying the gradient descent method to the unconstrained convex optimization problem \eqref{opt1}. Then, the following inequality holds:
\begin{equation*}
\langle\nabla f(x^k),\nabla f(x^{k-1})\rangle\leq|\nabla f(x^{k-1})|^2.
\end{equation*}\qed
\end{lemma}

\begin{lemma}(Theorem 2.1.5 in \cite{Nesterov2013})\label{lipschitz2}
Let $\mathcal{C}\subset\mathbb{R}^d$ be a closed convex set.
If a function $f\colon\mathcal{C}\to\mathbb{R}$ is convex and $L$-smooth, then for any $x,y\in\mathcal{C}$,
\begin{equation*}
f(x)-f(y)-\langle\nabla f(y),x-y\rangle\geq\frac{1}{2L}|\nabla f(x)-\nabla f(y)|^2.
\end{equation*}\qed
\end{lemma}

\section{Proof of Theorem \ref{wadgd}}
First, we show that an inequality containing $\|x^k-x^*\|^2$ holds, from which a telescoping sum with respect to $k$ can be taken.

\begin{lemma}
If the stepsize $t_k$ satisfies \eqref{swadgdstep}, the sequence $\{x^k\}$, a solution $x^*$, and the optimal value $f_*$ satisfy the following inequality;
\begin{equation}\label{neq4}
  \begin{split}
    \|x^{k+1}-x^*\|^2+&\|x^{k+1}-x^k\|^2+t_k(2+3\theta_k)(f(x^k)-f_*)\\
    &\leq\|x^k-x^*\|^2+\|x^k-x^{k-1}\|^2+3t_k\theta_k(f(x^{k-1})-f_*)+\gamma_k
  \end{split}
\end{equation}
\end{lemma}
\begin{proof}
Similarly to lemma \ref{neq1}
\begin{align*}
    \|x^{k+1}-x^k\|^2&=t_k^2L_k^2\|x^k-x^{k-1}\|^2-t_k^2\|\nabla f(x^{k-1})\|^2+2t_k^2\left\langle\nabla f(x^k),\nabla f(x^{k-1})\right\rangle\\
    &=\left(t_k^2L_k^2-\frac{t_k^2}{2t_{k-1}^2}\right)\|x^k-x^{k-1}\|^2-\frac{t_k^2}{2}\|\nabla f(x^{k-1})\|^2+2t_k^2\left\langle\nabla f(x^k),\nabla f(x^{k-1})\right\rangle.
\end{align*}
From \eqref{wadgdstep2} and \eqref{wadgdstep}, we obtain
\[
t_k\leq\sqrt{\frac{1+\frac{\gamma_k}{\|x^{k}-x^{k-1}\|^2}}{[2t_{k-1}^2L_k^2-1]_+}}t_{k-1}.
\]
Squaring both sides and rearranging, we obtain
\begin{equation*}
    t_k^2L_k^2-\frac{t_k^2}{2t_{k-1}^2}\leq\frac{1}{2}\left(1+\frac{\gamma_k}{\|x^{k}-x^{k-1}\|^2}\right).
\end{equation*}
Using this and Lemma \ref{norm1},
\begin{align*}
    \|x^k-x^{k-1}\|^2&\leq\frac{1}{2}\|x^k-x^{k-1}\|^2+\frac{\gamma_k}{2}-\frac{t_k^2}{2}\|\nabla f(x^{k-1})\|^2+2t_k^2\left\langle\nabla f(x^k),\nabla f(x^{k-1})\right\rangle\\
    &\leq\frac{1}{2}\|x^k-x^{k-1}\|^2+\frac{3}{2}t_k^2\left\langle\nabla f(x^k),\nabla f(x^{k-1})\right\rangle+\frac{\gamma_k}{2}\\
    &=\frac{1}{2}\|x^k-x^{k-1}\|^2+\frac{3}{2}t_k\theta_k\left\langle\nabla f(x^k),x^{k-1}-x^k\right\rangle+\frac{\gamma_k}{2}\\
    &\leq\frac{1}{2}\|x^k-x^{k-1}\|^2+\frac{3}{2}t_k\theta_k\left(f(x^{k-1})-f(x^k)\right)+\frac{\gamma_k}{2},
\end{align*}
where the last inequality uses the fact that Lemma \ref{convex1} in the appendix holds for $x^k$ and $x^*$ due to the convexity of $f$.

Multiplying this by 2 and rearranging terms yields
\begin{equation}
  \|x^{k+1}-x^k\|^2\leq\|x^k-x^{k-1}\|^2-\|x^{k+1}-x^k\|^2+3t_k\theta_k(f(x^{k-1})-f(x^k))+\gamma_k.
\end{equation}
By adding this to equation \eqref{standard}, we obtain
\begin{equation*}
  \begin{split}
    \|x^{k+1}-x^*\|^2+\|x^{k+1}-x^k\|^2\leq&\|x^k-x^*\|^2-2t_k(f(x^k)-f_*)+t_k^2\|\nabla f(x^k)\|^2\\
    &+\|x^k-x^{k-1}\|^2-\|x^{k+1}-x^k\|^2+3t_k\theta_k(f(x^{k-1})-f(x^k))+\gamma_k
  \end{split}
\end{equation*}
Rearranging this by using equation \eqref{gd}, the following holds;
\begin{equation*}
  \begin{split}
    \|x^{k+1}-x^*\|^2+&\|x^{k+1}-x^k\|^2+t_k(2+3\theta_k)(f(x^k)-f_*)\\
    &\leq\|x^k-x^*\|^2+\|x^k-x^{k-1}\|^2+3t_k\theta_k(f(x^{k-1})-f_*)+\gamma_k.
  \end{split}
\end{equation*}
\end{proof}

Then, we show that the sequence $\{x^k\}$ is contained within a closed ball centered on a solution $x^*$ by taking the telescoping sum of \eqref{neq4}.
\begin{lemma}\label{wadgdbounds}
  If the stepsize $t_k$ satisfies \eqref{wadgdstep}, the sequence of iterations $\{x^k\}$ is bounded.
  Furthermore, for a solution $x^*$, $x^k\in B(x^*,R)$ holds, where $R$ is a constant determined by $x^*$, the initial values, and $\{\gamma_k\}$.
\end{lemma}
\begin{proof}
Taking the telescoping sum of \eqref{neq4}, we get
\begin{align}
  \|x^{k+1}-x^*\|^2&+\|x^{k+1}-x^k\|^2+2\sum_{i=1}^{k}c_i(f(x^i)-f_*)\nonumber\\
  &\leq\|x^1-x^*\|^2+\|x^1-x^0\|^2+3t_1\theta_1(f(x^0)-f_*)+\sum_{i=1}^k\gamma_i\label{wadgdneq2},
\end{align}
where $\{c_i\}$ is a sequence of real numbers defined as follows;
\begin{equation}\label{c}
c_i=
\begin{cases}
  t_i(2+3\theta_i)                      & \text{if $i=k$,} \\
  t_i(2+3\theta_i)-3t_{i+1}\theta_{i+1} & \text{if $1\leq i\leq k-1$.} 
\end{cases}
\end{equation}
Also, from \eqref{wadgdstep2} and \eqref{wadgdstep}, we have $t_k\leq\sqrt{\frac{2}{3}+\theta_{k-1}}t_{k-1}$.
Squaring both sides yields
\begin{equation*}
t_k^2\leq\left(\frac{2}{3}+\theta_{k-1}\right)t_{k-1}^2,
\end{equation*}
and dividing this by $t_{k-1}$ gives
\begin{equation}\label{wadgdstep3}
  t_k\theta_k\leq\left(\frac{2}{3}+\theta_{k-1}\right)t_{k-1}.
\end{equation}

Here, from \eqref{gd}, \eqref{standard}, and \eqref{wadgdstep3}, we have
\begin{align*}
    \|x^1-x^*\|^2&+\|x^1-x^0\|^2+3t_1\theta_1(f(x^0)-f_*)\\
    &\leq\|x^0-x^*\|^2+(3t_1\theta_1-2t_0)(f(x^0)-f_*)+2t_0^2\|\nabla f(x^0)\|^2\\
    &\leq\|x^0-x^*\|^2+3t_0\theta_0(f(x^0)-f_*)+2t_0^2\|\nabla f(x^0)\|^2.
\end{align*}
Furthermore, since $\sum_{k=1}^{\infty}\gamma_k$ is bounded by assumption, using a certain real number $S>0$, the following inequality holds for arbitrary $k\in\mathbb{N}$;
\begin{equation*}
  \sum_{i=1}^k\gamma_i\leq S.
\end{equation*}
Therefore, from \eqref{wadgdneq2},
\begin{equation*}
  \begin{split}
    \|x^{k+1}-x^*\|^2+&\|x^{k+1}-x^k\|^2+\sum_{i=1}^{k-1}c_i(f(x^i)-f_*)\\
    &\leq\|x^0-x^*\|^2+3t_0\theta_0(f(x^0)-f_*)+2t_0^2\|\nabla f(x^0)\|^2+S\\
    &=R^2
  \end{split}
\end{equation*}
holds (where we assume $R>0$).

Here,
\begin{itemize}
\item When $i=k$, $c_k=t_k(2+3\theta_k)\geq0$ since $t_k\geq0$ and $(2+3\theta_k)>0$.
\item When $1\leq i\leq k-1$, $c_i=t_i(2+3\theta_i)-3t_{i+1}\theta_{i+1}\geq0$ from equation \eqref{wadgdstep3}.
\end{itemize}
Thus, $c_i\geq0$ for $1\leq i\leq k$. Furthermore, since $f_*$ is the optimal value, $f(x^i)-f_*\geq0$ holds.

Therefore, the following inequality holds;
\begin{equation*}
  \|x^{k+1}-x^*\|^2\leq R^2.
\end{equation*}
Since this holds for any $k$, $\{x^k\}$ is bounded and $x^k\in B(x^*,R)$ holds.
\end{proof}

From this lemma, even if $f$ is locally smooth, by restricting our discussion within this closed ball, we can definitively treat it as smooth with an existing Lipschitz constant $L$.

As described in Section 3.1, the proof of convergence also requires that the stepsize is bounded below by a strictly positive constant, which can be shown using the Lipschitz constant $L$.
The next lemma demonstrates this fact.
\begin{lemma}\label{positive2}
If $t_k$ satisfies \eqref{wadgdstep}, it satisfies the following inequality;
\begin{equation}\label{minstep-2}
  t_k\geq\min\left\{t_0,\frac{\delta}{\sqrt{3}L}\right\}
\end{equation}
where $L$ is the Lipschitz constant of $\nabla f(x)$ within the closed ball $B(x^*,R)$ from Lemma \ref{wadgdbounds}.
\end{lemma}
\begin{proof}
  First, depending on the value of $l_k^{(2)}$, the possible interval for $t_k$ can be categorized into the following cases:
  \begin{description}
    \item[(Case 1)] When $l_k^{(2)}=\sqrt{\frac{2}{3}+\theta_{k-1}}t_{k-1}$, we have $\sqrt{\frac{2}{3}+\theta_{k-1}}t_{k-1}\leq\frac{\delta}{\sqrt{2}L_k}$ from \eqref{wadgdstep1}.
    Here, since 
    \[
    \frac{\delta}{\sqrt{2}L_k}<\frac{1}{\sqrt{2}L_k}<\frac{t_{k-1}}{\sqrt{[2t_{k-1}^2L_k^2-1]_+}}\leq\sqrt{\frac{1+\frac{\gamma_k}{\|x^{k}-x^{k-1}\|^2}}{[2t_{k-1}^2L_k^2-1]_+}}t_{k-1}
    \]
    holds, we consequently have $u_k^{(2)}=\sqrt{\frac{2}{3}+\theta_{k-1}}t_{k-1}$.
    That is, since $l_k^{(2)}=u_k^{(2)}$, we obtain
    \begin{equation}\label{condition2-1}
    t_k=\sqrt{\frac{2}{3}+\theta_{k-1}}t_{k-1}.
    \end{equation}
    \item[(Case 2)] When $l_k^{(2)}=\frac{\delta}{\sqrt{2}L_k}$, we have $\frac{\delta}{\sqrt{2}L_k}\leq\sqrt{\frac{2}{3}+\theta_{k-1}}t_{k-1}$ from \eqref{wadgdstep1}.
    Since 
    \[
    \frac{\delta}{\sqrt{2}L_k}<\sqrt{\frac{1+\frac{\gamma_k}{\|x^{k}-x^{k-1}\|^2}}{[2t_{k-1}^2L_k^2-1]_+}}t_{k-1}
    \]
    also holds, we consequently obtain
    \begin{equation}\label{condition2-2}
      t_k\in\left[\frac{\delta}{\sqrt{2}L_k},u_k^{(2)}\right]
    \end{equation}
  \end{description}

  We prove the proposition using mathematical induction.
  \begin{itemize}
    \item When $k=1$, if $t_1$ falls into (Case 1), $t_1=\sqrt{\frac{2}{3}+\theta_0}t_0=t_0$ from \eqref{condition2-1} ($\because\theta_0=\frac{1}{3}$) . If $t_1$ falls into (Case 2), $t_1\geq\frac{\delta}{\sqrt{2}L_k}\geq\frac{\delta}{\sqrt{2}L}\geq\frac{\delta}{\sqrt{3}L}$ from \eqref{condition2-2}. Therefore, \eqref{minstep-2} holds in both cases.
    \item Assume that $t_{k-1}$ satisfies \eqref{minstep-2}. First, when $t_k$ satisfies (Case 2), it follows from \eqref{condition2-2} that $t_k \geq \frac{\delta}{\sqrt{2}L_k} \geq \frac{\delta}{\sqrt{3}L}$, which implies that \eqref{minstep-2} holds.

    On the other hand, when $t_k$ satisfies (Case 1), \eqref{condition2-1} yields $t_k = \sqrt{\frac{2}{3}+\theta_{k-1}}t_{k-1}$.
    Here, if $t_{k-1}$ also satisfies (Case 1), we have $t_{k-1} = \sqrt{\frac{2}{3}+\theta_{k-2}}t_{k-2}$ from \eqref{condition2-1}. Dividing both sides by $t_{k-2}$ yields $\theta_{k-1} = \sqrt{\frac{2}{3}+\theta_{k-2}} \geq \sqrt{\frac{2}{3}}$.
    Consequently, the following inequality holds:
    \[
    t_k \geq \sqrt{\frac{2}{3}+\sqrt{\frac{2}{3}}}t_{k-1} \geq t_{k-1}.
    \]
    In this case, by the inductive hypothesis, $t_k$ also satisfies \eqref{minstep-2}.
    Furthermore, if $t_{k-1}$ satisfies (Case 2), we obtain $t_{k-1} \geq \frac{\delta}{\sqrt{2}L}$ from \eqref{condition2-2}.
    Thus,
    \[
    t_k \geq \sqrt{\frac{2}{3}} \cdot \frac{\delta}{\sqrt{2}L} \geq \frac{\delta}{\sqrt{3}L},
    \]
    which implies that $t_k$ satisfies \eqref{minstep-2}.
  \end{itemize}
  Consequently, \eqref{minstep-1} holds for any $k\in\mathbb{N}$.
\end{proof}

From these lemmas, the following theorem regarding convergence to a solution holds.
\begin{theorem}\label{swadgd}
Let $f\colon\mathbb{R}^d\rightarrow\mathbb{R}$ be convex with locally Lipschitz gradient $\nabla f$.
If the stepsize $t_k$ satisfies \eqref{wadgdstep}, the sequence $\{x^k\}$ is bounded, and its accumulation point is an optimal solution, where the initial values are $t_0>0$, $\theta_0\geq0$, and $x^0\in\mathbb{R}^d$.
\end{theorem}
\begin{proof}
  From Lemma \ref{wadgdbounds}, $\{x^k\}$ is bounded.
  Thus, an accumulation point exists.

  Here, since $B(x^*,R)$ is a convex set and $f$ is locally smooth, the mapping $f(x)$ with domain $x\in B(x^*,R)$ is L-smooth for a certain constant $L>0$.
  Furthermore, since $x^*,x^k\in B(x^*,R)$ from Lemma \ref{wadgdbounds}, Lemma \ref{lipschitz2} in Appendix holds, so \eqref{standard2} holds.
  Adding this to equation \eqref{neq5} yields
  \begin{equation*}
    \begin{split}
      \|x^{k+1}-x^*\|^2+&\|x^{k+1}-x^k\|^2+t_k(2+3\theta_k)(f(x^k)-f_*)+\frac{t_k}{L}\|\nabla f(x^k)\|^2\\
      &\leq\|x^k-x^*\|^2+\|x^k-x^{k-1}\|^2+3t_k\theta_k(f(x^{k-1})-f_*)+\gamma_k.
    \end{split}
  \end{equation*}

  Taking the telescoping sum in the same manner as Lemma \ref{wadgdbounds}, the following inequality holds;
  \begin{equation*}
    \|x^{k+1}-x^*\|^2+\|x^{k+1}-x^k\|^2+\sum_{i=1}^{k}c_i(f(x^i)-f_*)+\sum_{i=1}^k\frac{t_i}{L}\|\nabla f(x^i)\|^2\leq R^2
  \end{equation*}
  where $c_i$ is the same as in \eqref{c}.
  At this time, since $c_i\geq0$ and $f(x^i)-f_*\geq0$ similarly to Lemma \ref{wadgdbounds},
  \begin{equation*}
    \sum_{i=1}^k\frac{t_i}{L}\|\nabla f(x^i)\|^2\leq R^2
  \end{equation*}
  holds.
  Therefore, taking the limit as $k\to\infty$, the following inequality holds;
  \begin{equation*}
    \sum_{i=1}^{\infty}\frac{t_i}{L}\|\nabla f(x^i)\|^2\leq R^2.
  \end{equation*}
  From this, $\lim_{i\to\infty}t_i\|\nabla f(x^i)\|^2=0$ holds.
  At this time, since $t_i\geq\min\left\{t_0,\frac{\delta}{\sqrt{3}L}\right\}>0$ from Lemma \ref{positive2},
  \begin{equation*}
    \lim_{i\to\infty}\|\nabla f(x^i)\|=0
  \end{equation*}
  holds.
  
  For any accumulation point $\lim_{j\to\infty}x^{i_j}=\bar{x}$, we have $\lim_{j\to\infty}\|\nabla f(x^{i_j})\|=\|\nabla f(\bar{x})\|=0$, meaning that $\bar{x}$ is a solution.
\end{proof}

\end{document}